\documentclass[oneside,a4paper,12pt]{amsart} 
\usepackage{amssymb}
\usepackage{amscd}
\usepackage[all,cmtip]{xy}
\usepackage{verbatim}
\usepackage{lscape}
\usepackage{enumerate}
\usepackage{xcolor}
\usepackage{adjustbox}
\theoremstyle{plain}
\newtheorem{theorem}{Theorem}[section]
\newtheorem{lemma}[theorem]{Lemma}
\newtheorem*{theorem*}{Theorem}

\newtheorem{corollary}[theorem]{Corollary}
\newtheorem{proposition}[theorem]{Proposition}

\theoremstyle{definition}

\theoremstyle{remark}
\newtheorem*{remark}{Remark}

\newtheorem*{example}{Example}
 \DeclareMathOperator{\Ext}{Ext}

\DeclareMathOperator{\Coker}{Coker} 
\DeclareMathOperator{\tr}{tr} \DeclareMathOperator{\Stab}{Stab}
\DeclareMathOperator{\Res}{Res} \DeclareMathOperator{\Ind}{Ind}
\DeclareMathOperator{\Spec}{Spec} \DeclareMathOperator{\Hom}{Hom}
\DeclareMathOperator{\End}{End} 
 \DeclareMathOperator{\rad}{rad}
 
\DeclareMathOperator{\Ima}{Im} \DeclareMathOperator{\res}{res}

 \DeclareMathOperator{\Id}{Id}

\DeclareMathOperator{\reg}{reg}
\DeclareMathOperator{\hreg}{hreg}

\DeclareMathOperator{\Gl}{Gl}

\DeclareMathOperator{\Syl}{Syl}

\DeclareMathOperator{\pSing}{p-Sing}
\DeclareMathOperator{\op}{op}
\DeclareMathOperator{\depth}{depth}

\DeclareMathOperator{\colim}{colim}

\begin{document}

\title{The Module structure of a Group Action on a Ring}

\author{Peter Symonds}
\thanks{This work was supported by the Engineering and Physical Research Council (Grant number EP/V036017/1).}
\address{Department of Mathematics\\
         University of Manchester\\
     Manchester M13 9PL\\
     United Kingdom}
\email{Peter.Symonds@manchester.ac.uk}

\date{26/03/2024}

\begin{abstract}
Consider a finite group $G$ acting on a graded Noetherian $k$-algebra $S$, for some field $k$ of characteristic $p$; for example $S$ might be a polynomial ring. Regard $S$ as a $kG$-module and consider the multiplicity of a particular indecomposable module as a summand in each degree. We show how this can be described in terms of homological algebra and how it is linked to the geometry of the group action on the spectrum of $S$.
\end{abstract}

\maketitle

\section{Introduction}
\label{sec:intro}

Given a field $k$ of prime characteristic $p$, which in this introduction we assume to be algebraically closed, we consider a Noetherian graded $k$-algebra $S= \oplus_i S_i$ and the action of a finite group $G$ on it by graded algebra automorphisms. The main example is a polynomial ring $k[V]$, where $V$ is a finite dimensional $kG$-module; in representation-theoretic terms this is the symmetric algebra on the dual of $V$.

We are particularly interested in the decomposition of $S$ into indecomposable $kG$-summands. We already know that in the case of $k[V]$ only a finite number of isomorphism types of indecomposable summands can appear, although this is not true in general \cite[4.2]{Sy-Cech}. We show that, for any finite dimensional indecomposable $kG$-module $M$, the summands in such a decomposition that are isomorphic to $M$ can be parametrised by a finite graded module over the invariant subalgebra $S^G$, which we call $\Hom^{\oplus}_{kG}(M,S)$. In particular, the multiplicity of $M$ as a summand of $S_i$ is equal to $\dim_k \Hom^{\oplus}_{kG}(M,S)_i$. This allows us to attach to these multiplicities various invariants from commutative algebra, such as Krull dimension and depth.

Recall that one can associate to a finite dimensional indecomposable $kG$-module $M$ a vertex $P$ (up to conjugacy), which is a $p$-subgroup, and a source $W$, which is a $kP$-module. The inertia subgroup of $W$ is $I = \{ g \in N_G(P) \mid {}^g W \cong W \}$. 

\begin{theorem}(cf.\ Theorem~\ref{th:sum})
With the notation above, either $\Hom_{kG}^{\oplus}(M,k[V])=0$ or we have
\[
\dim \Hom_{kG}^{\oplus}(M,k[V]) \leq \dim_k V^P
\]
and
\[
\depth \Hom_{kG}^{\oplus}(M,k[V]) \geq \dim_k V^{I_p},
\]
where $I_p$ is a Sylow $p$-subgroup of $I$.
\end{theorem}

We also consider various other functors on $kG$-modules that we apply to $S$, all of them related in some way to $\Hom_{kG}(M,S)$, for example $\Ext^i_{kG}(M,S)$.

Of particular interest is the module of invariant summands $S^{\oplus G} = \Hom^{\oplus}_{kG}(k,S)$, which turns out to have a natural ring structure. We also consider the Brauer quotient $S^{[G]}= S^G/ \sum_{p | [G:H]} \tr^G_H S^H$, which is also a ring. It is known that if $P$ is a $p$-group and $X$ is a $P$-set then $k[X]^{[P]} \cong k[X^P]$ (the Brauer construction, see e.g.\ \cite{Bro}), which motivates the next result.

\begin{theorem}(cf.\ Theorem~\ref{t:equiv})
Both $S^{\oplus G}$ and $S^{[G]}$ have spectra homeomorphic to $\Spec(S)^P/N_G(P)$, where $P$ is a Sylow $p$-subgroup of $G$. In particular, they both have dimension $\dim \Spec(S)^P$.
\end{theorem}

In Section~\ref{s:notation} we describe the functors on $kG$-modules that we will consider, in particular $\Hom^{\oplus}_{kG}(M,-)$. Sections~\ref{s:fix}--\ref{s:functor} deal with general rings $S$, while Sections~\ref{s:depth}--\ref{s:sum} are mostly concerned with polynomial rings. We conclude in Section~\ref{s:examples} with a couple of simple examples.
 
\section{Notation}
\label{s:notation}

Let $k$ be a field of characteristic $p$. We allow $p=0$, but most of our results are either known or trivial in this case (cf.\ \cite{Stanley}). All groups will be finite and $G_p$ will denote a Sylow $p$-subgroup of the group $G$ (and $G_0=1$). 

Let $R$ be a commutative graded $k$-algebra. Let $G$ be a finite group and let $S$ be a graded $RG$-module.  We are most most concerned with the case when $S$ is a $k$-algebra and $G$ acts on it by algebra automorphisms; $R$ will be a sub-algebra of the invariant sub-algebra $S^G$. The case about which we have most to say is when $S=k[V]$ for some finite dimensional $kG$-module $V$.

Let $M$ be a finite dimensional indecomposable $kG$-module. We wish to describe the part of $S$ that is in some way related to $M$. Note that $E_M=\End_{kG}(M)$ is a local ring and $\overline{E}_M:=E_M/\rad(E_M)$ is a division $k$-algebra.

First we consider $\Hom_{kG}(M,S)$. In the case when $S=k[V]$, this is isomorphic to the classical module of covariants $k[V,M^*]^G$, where $M$ is the $k$-dual of $M$. Note that $\Hom_{kG}(M,S)$ is naturally a right $E_M$-algebra. If $M$ is simple then $E_M$ is already a division algebra and $\dim_{E_M}\Hom_{kG}(M,S)_i$ is equal to the largest cardinal $r$ such that $M^{(r)}$ can be seen as a $kG$-submodule of $S_i$. We will be interested in the structure of $\Hom_{kG}(M,S)$ as a left $E_M^{\op}\otimes _k R$-module. If $M$ is simple and $P_M$ is its projective cover, then, for any graded $kG$-module $N$,  $\dim_k \Hom_{kG}(P_M,N)_i/ \dim_k E_M$, if finite, is equal to the number of times that $M$ appears in a composition series for $N_i$. (This is clear when $N$ is simple and also the functor $\Hom_{kG}(P_M,-)_i$ is exact, so use induction on the number of times $M$ occurs as a composition factor of $N_i$.)

Next we assume that $k$ has characteristic $p$ and we are given some set $\mathcal X$ of $p$-subgroups of $G$. Given $kG$-modules $M$ and $N$, define $\Hom_{kG}^{\mathcal X}(M,N)$ to be the quotient of $\Hom_{kG}(M,N)$ by the sum of the images of the usual transfer maps $\tr^G_P$ from $\Hom_{kP}(M,N)$ for all $P \in \mathcal X$. When $\mathcal X = \{ 1 \}$, this is just the Tate $\Ext$-group $\widehat{\Ext}^0_{kG}(M,N)$. We will sometimes write $\widehat{\Hom}_{kG}(M,N)$ for $\widehat{\Ext}^0_{kG}(M,N)$ and $\hat{H}^0(N)$ for $\widehat{\Ext}^0_{kG}(k,N)$. When $\mathcal X$ consists of all the $p$-subgroups of $G$ of index divisible by $p$, then we write $\Hom_{[kG]}(M,N):=\Hom_{kG}^{\mathcal X}(M,N)$ and $N^{[G]} := \Hom_{kG}^{\mathcal X}(k,N)$, the Brauer construction on $N$. The value of $\Hom_{kG}^{\mathcal X}(M,N)$ does not change if we close $\mathcal X$ under $p$-subgroups and conjugation, and usually we will assume that this has been done.  

Finally, we want to identify the $kG$-summands of $S$ that are isomorphic to $M$; we assume that $M$ is finite dimensional and indecomposable. Define $J_G(M,S)$ to be the subset of $\Hom_{kG}(M,S)$ consisting of homomorphisms that are \emph{not} split injective. This is certainly closed under multiplication by an element of $R$; it is also closed under addition, because if $f,g \in J_G(M,S)$ and $f+g$ is split injective, say $s(f+g)=\Id_M$, then, since $\End_{kG}(M)$ is local, either $sf$ or $sg$ is an automorphism, so either $f$ or $g$ is split injective, a contradiction.

We define $\Hom^{\oplus}_{kG}(M,S)=\Hom_{kG}(M,S)/J_G(M,S)$; it is naturally an $R$-module. It is actually equal to  $\Hom_{kG}(M,S)/\rad \Hom_{kG}(M,S)$, where $\rad$ denotes the radical of the category of $kG$-modules \cite{ARS, Kelly}, but we will not use this fact. 

\begin{lemma}
\label{l:sum}
For any finite dimensional indecomposable $kG$-module $M$, we have:
\begin{enumerate}
\item
$\Hom^{\oplus}_{kG}(M,-)$ commutes with direct sums.
\item
$\Hom^{\oplus}_{kG}(M,M) = \overline{E}_M$ and $\Hom^{\oplus}_{kG}(M,-)$ is naturally a right $\overline{E}_M$-module.
\item
If $X$ is a  $kG$-module with no summand isomorphic to $M$ then $\Hom^{\oplus}_{kG}(M,X) = 0$.
\end{enumerate}

\end{lemma}

\begin{proof}
$\Hom_{kG}(M,-)$ commutes with direct sums, because $M$ is finite dimensional. Let $S=\oplus T_{\lambda}$, a direct sum of indecomposable submodules, and let $p_{\lambda}$ denote projection onto $T_{\lambda}$. If $M \xrightarrow{f} S$ is not split injective then no $M \xrightarrow{p_{\lambda}f} T_{\lambda}$ can be split injective; since $M$ is finite dimensional, only a finite number of the $p_{\lambda}f$ are non-zero and $f$ is the sum of these. Thus $J_G(M,-)$ commutes with direct sums too.

The second part is because the radical of $E_M$ consists of the non-automorphisms, which correspond to the elements of $J_G(M,M)$.

For the third part, notice that clearly  $\Hom_{kG}(M,N) = J_G(M,N)$ in this case.
\end{proof}

The following proposition concerns infinite dimensional modules. The reader who is only interested in rings that are finite dimensional in each degree can safely skip the proof, since it is an easy consequence of the Krull-Schmidt Theorem in that case. But even in such a context, infinite dimensional modules can arise naturally, for example upon localising. We thank Mike Prest for his assistance with this material.

\begin{proposition}
\label{p:decomp}
Given a finite dimensional indecomposable $kG$-module $M$, any $kG$-module $X$ can be decomposed as $X=Y \oplus Z$, where $Y$ is isomorphic to a sum of copies of $M$ and $Z$ has no summands isomorphic to $M$.
\end{proposition}

\begin{proof} The module $M$ is finite dimensional, so it has finite endolength and hence is  $\Sigma$-pure-injective \cite[4.4.5 and example afterwards]{Prest}.

Consider all the submodules $U$ of $X$ that are isomorphic to a sum of copies of $M$ and are purely embedded in $X$. Give them a partial ordering by inclusion. By the definition of $\Sigma$-pure-injective, it follows that each $U$ is pure-injective. By the definition of pure-injective and the fact that $U$ is purely embedded in $X$, it follows that $U$ is a summand of $X$ and hence, if $U \leq U' \leq X$, then $U$ is a summand of $U'$.

  Consider any chain in this ordering.  The union of these submodules is isomorphic to a sum of copies of $M$ and, since a direct limit of pure embeddings is a pure embedding \cite[2.1.2]{Prest}, the union is  purely embedded in $X$. By Zorn's Lemma we obtain a submodule $Y \leq X$ that is a maximal subject to being purely embedded and isomorphic to a sum of copies of $M$. By the definition of $\Sigma$-pure-injective, it follows that $Y$ is pure-injective. By the definition of pure-injective and the fact that $Y$ is purely embedded in $X$, it follows that $Y$ is a summand of $X$. We take $Z$ to be a complement.

By the maximality of $Y$, it follows that $Z$ can have no summand isomorphic to $M$.
\end{proof} 

\begin{lemma}
\label{l:summands}
$\dim_{\overline{E}_M} \Hom^{\oplus}_{kG}(M,S)$ is equal to the largest cardinal $r$ such that $M^{(r)}$ can be seen as a $kG$-summand of $S$. If $S$ is graded then $\dim_{\overline{E}_M} \Hom^{\oplus}_{kG}(M,S)_i$ is equal to the largest cardinal $r$ such that $M^{(r)}$ can be seen as a $kG$-summand of $S_i$.
\end{lemma}

\begin{proof}
Use Proposition~\ref{p:decomp} and Lemma~\ref{l:sum}.
\end{proof}

The summands isomorphic to $M$ from Proposition~\ref{p:decomp} can be realised as follows. Take a basis for $\Hom^{\oplus}_{kG}(M,S)$ over $\dim_{\overline{E}_M}$ and lift the elements to $\Hom_{kG}(M,S)$. The images of $M$ in $S$ under each of these form the summands of $S$ isomorphic to $M$.

Let $K_G(M,S) \leq \Hom_{kG}(M,S)$ be the $k$-subspace spanned by the  homomorphisms that factor through some finite dimensional indecomposable $kG$-module that is \emph{not} isomorphic to $M$. This is clearly an $R$-submodule.

\begin{lemma}
\label{l:JK}
We have $J_G(M,S)=K_G(M,S)$.
\end{lemma}

\begin{proof}
First we show that $K_G(M,S) \leq J_G(M,S)$. Suppose that $M \xrightarrow{f} S$ in $K_G(M,S)$ factors as $M \xrightarrow{u} N \xrightarrow{v} S$, with $N$ indecomposable and not isomorphic to $M$. If $f$ is split injective then so is $u$, so $M$ is isomorphic to a summand of $N$, a contradiction.

Now we show that $J_G(M,S) \leq K_G(M,S)$.  Suppose that $M \xrightarrow{f} S$ is not split injective. Let $L \leq S$ be a finite dimensional $kG$-submodule containing $\Ima(f)$. Write $L = \oplus _iL_i$ as a sum of indecomposables and let $p_i$ denote projection onto $L_i$. Then $f = \sum_i p_if$ is in $K_G(M,S)$ unless some $L_i$ is isomorphic to $M$. Even then, if $p_if$ is not surjective it must factor through its image, which cannot be isomorphic to $M$. Thus $f \in K_G(M,S)$ unless some $p_if$ is an isomorphism and so $M \xrightarrow{f} L$ is split.

We see that $f \in K_G(M,S)$ unless $f$ is injective and $f(M)$ is a summand of every finite dimensional $kG$-submodule of $S$ containing it. But $S$ is the direct limit of the its finite dimensional submodules containing $f(M)$, and a direct limit of split embeddings is pure \cite[2.1.3]{Prest}. Thus $f$ is a pure embedding and since $M$ is pure injective (cf.\ the proof of Proposition~\ref{p:decomp}), $f$ is split.
\end{proof}

Of particular interest is $S^{\oplus G} :=\Hom^{\oplus}_{kG}(k,S)$, which corresponds to the trivial summands of $S$. If $S$ is a ring then so is $S^{\oplus G}$, by the product that sends $f$ and $g$ to 
\[k \rightarrow k \otimes_k k \xrightarrow{f \otimes g} S \otimes _k S \rightarrow S. \]
Notice that if $f$ factors through $X$, say, then $f \otimes g$ factors through $ X \otimes k $,  hence $J_G(k,S)$ is an ideal and the product is defined on the quotient. Similarly, any $\Hom^{\oplus}_{kG}(M,S)$ is naturally a module over $S^{\oplus G}$ and the action of $R$ factors through it.

If $M$ is an indecomposable $kG$-module with vertex not in $\mathcal X$, then there is a natural map $\Hom_{kG}^{\mathcal X}(M,S) \rightarrow \Hom^{\oplus}_{kG}(M,S)$. This is because if $f\! : M \rightarrow S$ is in the image of $\tr^G_P$ for some $P \in \mathcal X$ then $f$ factors through the module $\Ind^G_P \Res^G_P M$, which is projective relative to $\mathcal X$ and so cannot contain $M$ as a summand. In particular, there is a homomorphism of $k$-algebras $S^{[G]} \rightarrow S^{\oplus G}$ and the map $R \rightarrow S^{\oplus G}$ factors through it. Similarly, $\Hom_{kG}^{\mathcal X}(k,S)$ is a ring and $\Hom_{kG}^{\mathcal X}(M,S)$ is a module over it.

One of the main features of these definitions is that we always obtain an $R$-module,  so we can associate to it various invariants from commutative algebra.

We cannot, in general, expect $\Hom^{\oplus}_{kG}(M,S)$ to correspond to a submodule of $S$ as an $R$-module. The best we can do is to use the notion of a structure theorem from \cite{Sy4}. 

We choose a graded Noether normalization $k[d_1, \ldots , d_n]$ for $R$. Then, by \cite[3.1]{Sy4} and Lemma~\ref{l:summands}, $\Hom^{\oplus}_{kG}(M,S)$ has a structure theorem, meaning that there is a finite indexing set $J$ and for each $j \in J$ a subset $I_j \subseteq \{ 1, \ldots , n \}$ and a homogeneous element $f_j \in \Hom^{\oplus}_{kG}(M,S)$ such that the multiplication map $\oplus _{j \in J} (k[d_i]_{i \in I_j} \otimes _k f_j \overline{E}_M ) \rightarrow \Hom^{\oplus}_{kG}(M,S)$ is an isomorphism of right $\overline{E}_M$-modules.

Lift $f_j$ to a homogeneous element of $\Hom_{kG}(M,S)$ and let $M_j$ be its image. Then the multiplication map $\oplus_{j \in J} k[d_i]_{i \in I_j} \otimes _k M_j \rightarrow S$ is injective and its image is a $kG$-summand of $S$. It is maximal as a $kG$-summand that is a sum of copies of $M$. 

If we sum the left hand side over all isomorphism classes of indecomposable $kG$-summands of $S$, we obtain a $kG$-isomorphism with $S$. This is a structure theorem in the sense of \cite[1.1]{Sy4} when the number of isomorphism classes of indecomposable $kG$-summands of $S$ is finite.

\begin{lemma}
\label{l:Rmod}

If $R$ is Noetherian and $S$ is finite over $R$ then so are the $\Hom_{kG}^{\mathcal X}(M,S)$ for any $\mathcal X$, $\Hom^{\oplus}_{kG}(M,S)$, $S^{[G]}$ and $S^{\oplus G}$.
\end{lemma}

\begin{proof}
Since $M$ is finite dimensional, $\Hom_{k}(M,S)$ is finite over $R$. Now $\Hom_{kG}(M,S)$ is a submodule of this and, since $R$ is Noetherian, it is also finite over $R$. All the modules mentioned are quotients of this for various $M$.
\end{proof}

\section{Fixed Points}
\label{s:fix}

This section is based on the work of Feshbach \cite{Feshbach} and Fleischmann \cite{Fleischmann} and many subsequent authors, although these sources only consider polynomial rings. 

Let $k$ be a field of characteristic $p$ and let $S$ be a commutative Noetherian $k$-algebra. Let $G$ be a finite group that acts as algebra automorphisms of $S$, preserving the grading if $S$ is graded, and set $R=S^G$, also a Noetherian algebra. Then $G$ acts on $\Spec(S)$ and we define $\Spec(S)^G$ to be the set of prime ideals $J<S$ that are fixed as sets and for which the action of $G$ on $S/J$ is trivial. When $k$ is algebraically closed and we consider only maximal ideals, this is just the set of maximal ideals fixed under $G$. Let $I(G,S)$ be the ideal of $S$ generated by the elements of the form $(g-1)s$ for $g \in G$, $s \in S$. Then, by construction, $S/I(G,S)$ is the largest quotient ring of $S$ on which $G$ acts trivially and hence $\Spec ( S/I(G,S)) \cong \Spec (S)^G$. 

Given a set $\mathcal Y$ of subgroups of $G$, define $I( \mathcal Y,S)= \cap_{ H  \in  \mathcal Y} I(H,S)$ (and $I(\emptyset,S)=S$). Notice that this does not change if we close $\mathcal Y$ under supergroups.

By basic properties of $\Spec$, we obtain
\[ \Spec ( S/I(\mathcal Y,S)) \cong \cup_{ H \in \mathcal Y} \Spec (S)^H. \]

Now let $\mathcal X$ be a set of subgroups of $G$. Let $T_G(\mathcal X,S):= \sum_{H \in \mathcal X} \tr_H^G(S^H)$, an ideal in $S^G$, and let $ST_G(\mathcal X,S)$ denote the ideal it generates in $S$. These are unchanged if we close $\mathcal X$ under conjugation and subgroups. They are also unchanged if we replace each $H \in \mathcal X$ by one of its Sylow $p$-subgroups, so we will normally only consider classes of $p$-subgroups. For $H \leq G$ we write $\mathcal S_p(H)$ for the set of $p$-subgroups of $H$ and  
$T_H(\mathcal X,S)$ for $T_H(\mathcal X \cap \mathcal S_p(H),S)$.

\begin{theorem}
\label{t:inc}

Let  $\mathcal X \subseteq \mathcal S_p(G)$ be closed under subgroups and conjugation. 
\begin{enumerate}
\item
There is an inclusion of ideals in $S$, $ST_G(\mathcal X,S) \leq I(\mathcal S_p(G) \smallsetminus \mathcal X,S)$. 
\item
The associated surjection of rings $S/ST_G(\mathcal X,S) \rightarrow S/I(\mathcal S_p(G) \smallsetminus \mathcal X,S)$ has nilpotent kernel.
\item
The surjections of rings $R/T_G(\mathcal X,S) \rightarrow   R/(R \cap ST_G(\mathcal X,S)) \rightarrow R/(R \cap I(\mathcal S_p(G) \smallsetminus \mathcal X,S))$ have nilpotent kernels.
\end{enumerate}
\end{theorem}

\begin{proof}

For part (1), let $H \in \mathcal X$, $x \in S^H$ and consider $\tr_H^Gx$; we need to show that $\tr_H^Gx \in I(\mathcal S_p(G) \smallsetminus \mathcal X,S)$.  The result is trivially true when $\mathcal X = \mathcal S_p(G)$, so we may assume we have some $K \in \mathcal S_p(G) \smallsetminus \mathcal X$. By the Mackey formula, $\tr_H^Gx = \res _K^G \tr_H^Gx$ is equal to $\sum_{g \in K \backslash G/H} \tr^K_{K \cap {}^g \! H} gx$.  

We cannot have $K= K \cap {}^g \! H$, since $\mathcal X$ is closed under subgroups and conjugation and $K \not \in \mathcal X$, so $|K:K \cap {}^g \! H|gx=0$. It follows that $\tr^K_{K \cap {}^g \! H} gx=\tr^K_{K \cap {}^g \! H} gx - |K:K \cap {}^g \! H|gx = \sum_{j \in K/K \cap {}^g \! H} (j-1) gx \in I(K,S)$. This is true for any  $K \in \mathcal S_p(G) \smallsetminus \mathcal X$, which completes the proof of part (1).

For part (2), by the Nullstellensatz, it is enough to check that the quotients by the two ideals are equal modulo any maximal ideal $N$ of $S$. Our constructions are compatible with an ideal invariant under $G$, so  we let $K=\Stab_G(N)$, $M = \cap_{g \in G/K}gN$ and consider $\overline{S}=S/M$, which is isomorphic to  $\prod_{g \in G/K}S/gN$ by the Chinese Remainder Theorem. This is a product of fields, with spectrum a $G$-set isomorphic to $G/K$; let $x$ be the point $1K$. Our constructions also commute with extension of the field $k$, so we may assume that $K$ acts trivially on $S/N$. 

By the discussion on fixed-point sets above, $x \in X=\Spec(\overline{S}/ I(\mathcal S_p(G) \smallsetminus \mathcal X,\overline{S}))$ if and only if there is a subgroup $H \in S_p(G) \smallsetminus \mathcal X$ with $Hx=x$. This happens if and only if $\Syl_p(K) \not \subseteq \mathcal X$. Thus $ I(\mathcal S_p(G) \smallsetminus \mathcal X,\overline{S})$ is 0 if $\Syl_p(K) \not \subseteq \mathcal X$ and is equal to $\overline{S}$ otherwise.

In the case of $ST_G(\mathcal X,S)$, it is not immediate that the image of ${S}T_G(\mathcal X,{S})$ in $\overline{S}$ is $\overline{S}T_G(\mathcal X,\overline{S})$. First we show that $\overline{S}T_G(\mathcal X,\overline{S}) = \overline{S}I(\mathcal{S}_p(G) \smallsetminus \mathcal X,\overline{S})$. The ring $\overline{S}$ is a product of fields permuted transitively by $G$, so $\overline{S}T(\mathcal{S}_p(G) \smallsetminus \mathcal X,\overline{S})$ is either 0 or the whole of $\overline{S}$. By part (1), if $I(\mathcal S_p(G) \smallsetminus \mathcal X,\overline{S})=0$ then $\overline{S}T_G(\mathcal X,\overline{S})=0$ and we are done. 

If $I(\mathcal S_p(G) \smallsetminus \mathcal X,\overline{S}) = \overline{S}$ then $\Syl_p(K) \subseteq \mathcal X$, so let $P \in \mathcal X$ be a Sylow $p$-subgroup of $K$.
Let $e_x \in \overline{S}$ denote the idempotent corresponding to $x$.
We have $\tr^G_Pe_x = \tr^G_K\tr^K_Pe_x=\tr^G_K|K:P|e_x=|K:P| \tr^G_Ke_x.$ The elements in the sum for $\tr^G_Ke_x$ are the distinct primitive idempotents, so this sum is equal to 1. Also $|K:P| \ne 0$ in $k$, and we deduce that $\overline{S}T_G(\mathcal X,\overline{S})= \overline{S}$ in this case, hence $\overline{S}T_G(\mathcal X,\overline{S}) = \overline{S}I(\mathcal{S}_p(G) \smallsetminus \mathcal X,\overline{S})$, as claimed.

Now $e_x$ can be lifted to $\tilde{e}_x \in S$, although this $\tilde{e}_x$ might not be invariant under $K$. But $e_x = e_x^{|K|}$ is the image of $\mathcal N _K \tilde{e}_x = \prod_{k \in K} g \tilde{e}_k \in S^K$ and thus $\tr^G_K e_x$ and hence $\overline{S}$ are in the image of $T_G(\mathcal X,S)$, as required. 

For the second arrow in part (3) we again check modulo a maximal ideal $M'$ of $R$. Let $N$ be a maximal ideal of $S$ containing $SM'$; with the same notation as before, $SM' \leq \cap_{g \in G/K} gN=M$. Thus $M' \leq R \cap SM' \leq R \cap M$, but $M'$ is maximal in $R$, so $M'=R \cap M$. The natural homomorphism $R/M' \rightarrow S/M$ is injective and the result follows from part (2). 

For the first arrow, let $u=\sum_i s_i t_i \in R \cap ST_G(\mathcal X,S)$, where $s_i \in S$, $t_i \in T_G(\mathcal X,S)$. Then, since $u \in R$,  $u^{|G|} = \mathcal N_G (u)=  \prod_{g \in G} gu= \prod_{g \in G} (gs_i)t_i$ is a polynomial in the $t_i$ with coefficients that are symmetric functions in the $gs_i$, hence in $S^G=R$. Thus $u^{|G|} \in T_G(\mathcal X,S)$.
\end{proof}

\begin{example} The two main examples are the extreme ones.
\begin{enumerate}
\item 
When $\mathcal X = \{ 1 \}$, we see that $\Spec (S/I(\mathcal S_p(G) \smallsetminus \{1\}, S))$ is the $p$-singular set $\pSing _G(\Spec (S))$ of $\Spec (S)$, that is, it consists of the prime ideals fixed (in the sense mentioned above) by a non-trivial $p$-element of $G$. $\Spec (R/(R \cap I( \mathcal S_p(G) \smallsetminus \{1\}, S))$ is the image $\pSing_G(\Spec(S))/G$ of this in $\Spec(S)/G$. In this case, $R/T_G(\{ 1 \},S) \cong \hat{H}^0(G,S)$ and we obtain $\dim \hat{H}^0(G,S) = \dim \pSing_G(\Spec(S))$.

\item 
When $\mathcal X$ consists of all the $p$ subgroups that are not Sylow $p$-subgroups, we see that $\Spec (S/I( \Syl_p(G),S))$ consists of the prime ideals that are fixed by some Sylow $p$-subgroup. $\Spec (R/(R \cap I( \Syl_p(G),S)))$ is the image of this in $\Spec(S)/G$, which is homeomorphic to $\Spec(S)^P/N_G(P)$, where $P$ is a Sylow $p$-subgroup of $G$.  Note also that $R/T_G(\mathcal S_p(G) \smallsetminus \Syl_p(G),S) = S^{[G]}$; thus $\dim S^{[G]} = \dim \Spec(S)^P$.

\end{enumerate}
\end{example}

\begin{theorem}
\label{t:equiv}

Let $G$ be a finite group and set $T=T_G(\mathcal S_p(G) \smallsetminus \Syl_p(G) ,S)$ and $I=I( \Syl_p(G),S)$. There are natural homomorphisms of algebras:

\[
\xymatrix{
S^{[G]} \ar@{->>}^{\iota_1}[r] \ar@{->>}_{\iota_2}[d]
& S^{\oplus G} \ar@{->>}^{\iota_3}[d] \\
R/(R \cap ST) \ar@{->>}^{\iota_4}[r]
\ar@{^{(}->}_{\rho_1}[d]
&R/(R \cap I) \ar@{^{(}->}^{\rho_2}[d] \\
(S/ST)^G \ar@{->>}^{\iota_5}[r] 
& (S/I)^G.
}
\]
The homomorphisms labelled $\iota$ have nilpotent kernels. All six induce universal homeomorphisms on spectra.
\end{theorem} 

By definition, a morphism of schemes is a universal homeomorphism if it is a homeomorphism and remains so under any base change.  A homorphism of rings inducing a universal homeomorphism is also called an $F$-isomorphism or an inseparable isogeny in the literature. When the rings are both annihilated by the prime $p$ this is equivalent to the kernel being locally nilpotent and any element in the codomain having some $p^n$-power in the image \cite[24.46.9]{Stacks}. 

\begin{proof}
Note that $S^{[G]} = R/T$. The homomorphism $\iota_1 $ was introduced after Lemma~\ref{l:JK} and $\iota_4$ and $\iota_5$ were described in Theorem~\ref{t:inc}. The others are clear, except for $\iota_3$. For this, consider a $kG$-homomorphism $k \xrightarrow{f} S$ that factors as $k \xrightarrow{u} N \xrightarrow{v} S$, with $N$ indecomposable and not isomorphic to $k$. Let $N' \leq N$ be the submodule generated by elements of the form $(h-1)n$, where $h\in G$ is a $p$-element and $n \in N$. Suppose that $u(k) \not \leq N'$; then $u(k)$ has non-zero image in $N/N'$. But $N/N'$ is the largest quotient of $N$ on which every $p$-element of $G$ acts trivially, so the action of $G$ factors through a $p'$-group and $N/N'$ must be semisimple. Thus $u(k)$ is a summand and it follows that $u(k)$ is also a summand of $N$, a contradiction. Thus $u(k) \leq N'$ and hence $f(k) \subseteq I$.  

The nilpotence of the kernels of $\iota_2$, $\iota_4$ and $\iota_5$ was proved in Theorem~\ref{t:inc}.  The result for $\iota_1$ and $\iota_3$ now follows automatically.  

In order to show that $\rho_2$ is a universal homeomorphism, let $P$ be a Sylow $p$-subgroup of $G$ and consider $s + I \in (S/I)^G$. The action of $G$ on $S/I$ is trivial by definition, so $(s+I)^{|P|} = \mathcal N_P (s+I)=|G:P|^{-1} \tr_P^G \mathcal N_P(s+I) = |G:P|^{-1} \tr_P^G \mathcal N_P(s)+I = \rho_2 (|G:P|^{-1} \tr_P^G \mathcal N_P(s)+ R \cap I)$, since $\tr^G_P \mathcal N_P(s) \in R$. 

The result for $\rho_1$ follows immediately, since the other homomorphisms in the bottom square induce universal homeomorphisms. For a more ring-theoretic proof, the result for $\rho_2$ implies that for any $s +ST\in (S/ST)^G$ there is an $r \in R$ such that $s+ST - \rho_1(r+R \cap ST)$ is in the kernel of $\iota_5$. Thus, for some large $p^m$, we have $0 = (s+ST - \rho_1(r+R \cap ST))^{p^m}=(s+ST)^{p^m} - \rho_1((r+R \cap ST)^{p^m})$.
\end{proof}

\section{Functors}
\label{s:functor}

Let $F$ be a $k$-linear functor from $kG$-modules to $k$-vector spaces that commutes with infinite direct sums. We have in mind $F(X)=\Hom_{kG}(M,X)$, for some finite dimensional $kG$-module $M$, or some functor related to it as in Section~\ref{s:notation}.

We can extend this to graded modules in the usual way. If we consider $RG$-modules $X$, where $R$ is a $k$-algebra with trivial $G$-action, then $F(X)$ is naturally an $R$-module: if $m_r$ represents multiplication by $r \in R$ on $X$ then $F(m_r)$ represents the action of $r$ on $F(X)$.

Since $F$ commutes with direct sums and the action of $G$ on $R$ is trivial we have $F(R \otimes_k M) \cong R \otimes_k F(M)$ as $R$-modules. 

When we wish to record the group $G$ involved in the definition of the functor $F$ we write $F_G$. In many cases $F_?$ is naturally a cohomological Mackey functor; thus $F_G$ is naturally a $k$-summand of $F_P$, where $P$ is a Sylow $p$-subgroup of $G$ (the inverse to the restriction $F_P \rightarrow F_G$ is $|G:P|^{-1}\tr^G_P$). If we evaluate on $RG$-modules it is even an $R$-summand. Note that, in general, if an $R$-module $X$ is a summand of $Y$, then $\dim X \leq \dim Y$ and $\depth X \geq \depth Y$. Therefore $\dim F_G(X) \leq \dim F_P(X)$ and $\depth F_G(X) \geq \depth F_P(X)$ (provided $F_G(X) \ne 0$).

\begin{remark} \hspace{-1pt}
\begin{enumerate} 
\item
There does not seem to be a trouble-free definition for the depth of the zero module. A value of $+\infty$ would be consistent with most of our results, but we will treat this case separately in our statements.
\item
We should really record the ring and write $\depth_R$, but this should always be clear from the context.
\item
At this point we do not know that $F(X)$ is finite over $R$, so the depth should be defined in terms of vanishing of local cohomology rather than regular sequences (cf.\ \cite[3.5]{Totaro}).
\end{enumerate}
 
\end{remark}

In a similar vein we have the following result.

\begin{proposition}
\label{prop:vertex}
Let $M$ be an indecomposable $kG$-module and with vertex $P$ and source $W$. Then $ \dim \Hom_{kG}^{\mathcal X}(M,S) \leq \dim \Hom_{kP}^{\mathcal X}(W,S)$ and $ \depth \Hom_{kG}^{\mathcal X}(M,S) \geq \depth \Hom_{kP}^{\mathcal X}(W,S)$ (assuming $\Hom_{kG}^{\mathcal X}(M,S) \ne 0$).
\end{proposition}

\begin{proof}
We know that $\Hom_{kG}^{\mathcal X}(M,S)$ is a summand of $\Hom_{kG}^{\mathcal X}(W \! \uparrow _P^G,S)$, which in turn is isomorphic to $\Hom_{kP}^{\mathcal X}(W,S)$, all as $S^G$-modules. The result follows.
\end{proof}

\begin{corollary}
\label{cor:bound}
Let $M$ be an indecomposable $kG$-module with vertex $P$ and let $\mathcal X$ be a class of $p$-subgroups of $G$, closed under conjugation and subgroups, that does not contain $P$. We have
\begin{enumerate}
\item $\dim \Hom^{\mathcal X}_{kG}(M,S) \leq \dim  \Hom^{\mathcal X}_{kP}(k,S) \leq \max_{Q \in \mathcal S_p(P) \setminus \mathcal X} \dim \Spec(S)^Q$ and
\item
$ \dim \Hom_{kG}^{\oplus} (M,S)  \leq \dim \Spec(S)^{G_p}$.
\end{enumerate}
\end{corollary}

\begin{proof}
For the left hand inequality in part (1), Proposition~\ref{prop:vertex} allows us to replace $\Hom^{\mathcal X}_{kG}(M,S)$ by $\Hom_{kP}^{\mathcal X}(W,S)$, where $W$ is a source of $M$. This is finite over $\Hom^{\mathcal X}_{kP}(k,S)$ by the discussion after Lemma~\ref{l:JK}.  The right hand inequality follows from the discussion on fixed points of spectra. 

For part (2), take $\mathcal X$ to be the set of all subgroups of $G$ that are conjugate to a proper subgroup of $P$. Then $\Hom_{kG}^{\oplus} (M,S)$ is a quotient of $\Hom^{\mathcal X}_{kG}(M,S)$ and we can apply part (1).
\end{proof}

Our functors also commute with filtered colimits.

\begin{lemma}
\label{l:colim}
If $M$ is finite dimensional then $\Hom_{kG}^{\mathcal X}(M,-)$,  $\Hom_{kG}^{\oplus}(M,-)$, $(-)^{[G]}$ and $(-)^{\oplus G}$ commute with filtered colimits.
\end{lemma}

\begin{proof} This is true for $\Hom_{kG}(M,-)$, because a finite dimensional $kG$-module is compact. For the other functors $F$ in the list above, we have a short exact sequence $I(-) \rightarrow \Hom_{kG}(M,-) \rightarrow F(-)$. A standard diagram chase reduces the problem to showing that the natural map $\colim I(X_{\lambda}) \rightarrow I(\colim X_{\lambda})$ is surjective. 

For $\Hom_{kG}^{\mathcal X}(M,-)$ this follows from elementary properties of the transfer. For $\Hom_{kG}^{\oplus}(M,-)$ we use the description in terms of $K_G(M,-)$ from Lemma~\ref{l:JK}. If $M \xrightarrow{f} \colim X_{\lambda}$ is in $K_G(M,\colim X_{\lambda})$ then it factors through some finite dimensional indecomposable module $N$ that is not isomorphic to $M$. The map $N \xrightarrow{f} \colim X_{\lambda}$ will factor through some $X_{\lambda}$, so $f$ is in the image of $K_G(M,X_{\lambda})$.
\end{proof}

If $F$ commutes with  filtered colimits then it commutes with localisation: $F(M \otimes _R R[\tfrac{1}{r}]) \cong F(M) \otimes _R R[\tfrac{1}{r}]$ for $r \in R$. 

A basic tool here is the \v{C}ech complex, strictly speaking the  ``extended \v{C}ech complex'' or the ``stable Koszul complex''.

Given a ring $R$, a sequence of homogeneous elements $\mathbf{x}= x_1, \ldots , x_r$ from $R$ and an $R$-module $M$, the \v{C}ech complex $\check{C} (\mathbf{x};M)$ is a cochain complex
\[
M \rightarrow \bigoplus _i M_{x_i} \rightarrow \bigoplus_{i<j}M_{x_ix_j} \rightarrow \cdots \rightarrow M_{x_1 \cdots x_r},
\]
where $M_x$ denotes the localization obtained by inverting $x$. It can be obtained as follows. $\check{C} (x_i;R)$ is the complex
\[
R \rightarrow R_{x_i},
\]
with $R$ in degree 0 and $R_{x_i}$ in degree 1
and
\[
\check{C} (\mathbf{x};M)= \left( \bigotimes_{i=1}^r \! {}_R \; \check{C} (x_i;R) \right) \otimes_R M.
\]

Observe that if $G$ is a finite group and $M$ is an $RG$-module then $\check{C} (\mathbf{x};M)$ is a complex of graded $RG$-modules; up to isomorphism it is independent of the ordering of the elements of $\mathbf{x}$.

The $j$th homology group of this complex is equal to the local cohomology group $H^j_{(\mathbf{x})}(M)$, where $(\mathbf{x})$ denotes the ideal in $R$ generated by $\mathbf{x}$. It is known that if $\mathbf{y}$ is another sequence such that $\rad (\mathbf{y})= \rad (\mathbf{x})$ then $H^*_{(\mathbf{y})}(M)=H^*_{(\mathbf{x})}(M)$. We are particularly interested in the case when $\rad (\mathbf{x})= \mathfrak m = \mathfrak m_R := \rad(R_{>0})$, the radical of the ideal of elements in positive degrees (but this ideal is only maximal if $R_0/ \! \rad(R_0)$ is a field). For more information, see \cite[3.5]{BH}, \cite[3.1]{Totaro}.
 
For the functors $F$ above we have $\check{C} (\mathbf{x};F(X)) \cong F(\check{C} (\mathbf{x};X))$.

For a polynomial ring $k[V]$ we showed in \cite[4.2]{Sy-Cech} that the complex $\check{C} (\mathbf{x};k[V])$ is split as a complex of $kG$-modules in degrees greater than or equal to $-\dim_k V$. This has the following immediate consequence for the regularity or its variant $\hreg$, as in \cite[\S3]{Sy-Cech}.

\begin{proposition}
For any $k$-linear functor $F$ from $kG$-modules to $k$ vector spaces that commutes with filtered colimits, any finite dimensional $kG$-module $V$ (given grading $-1$, so $V^*$ has grading 1) and any subring $R \leq k[V]^G$ such that $k[V]$ is finite over $R$ we have $\hreg_RF(k[V]) \leq - \dim_k V$ and $\reg_R F(k[V]) \leq 0$.
\end{proposition}

The regularity is particularly useful for finding bounds on the degrees of the generators and relations in a minimal presentation of the $R$-module $F(k[V])$ (or a presentation as a $k$-algebra if it is a $k$-algebra), at least if $F(k[V])$ is finite over $R$, see \cite{Sy-reg}. In particular, $F(k[V])$ is completely determined if it is computed up to some known degree.

We also showed in \cite{Sy-Cech} that only a finite number of non-isomorphic indecomposable summands occur in $k[V]$ and that the localisations in the \v{C}ech complex are all direct sums of these.

\begin{lemma}
\label{l:structure}
If $S$ is a direct sum of finite dimensional $kG$-modules, for example $k[V]$, let $\{ M_\lambda \}_{\lambda \in \Lambda }$ be the set of isomorphism types of indecomposable summands. Then evaluation of homomorphisms yields a split surjection of $kG$-modules 
\[
\oplus_{\lambda \in \Lambda } \Hom_{kG}(M_\lambda,S) \otimes _{E_M} M_\lambda \rightarrow S.
\]
\end{lemma}

\begin{proof}
This is clearly true when $S$ is indecomposable and both sides commute with direct sums.
\end{proof}

\begin{proposition}
\label{p:finite}
Suppose that $S$ is an $RG$-module that is finite as an $R$-module and as a $kG$-module is a direct sum of finite dimensional modules of finitely many isomorphism types. Let $F$ be a $k$-linear functor that commutes with direct sums and takes finite dimensional $kG$-modules to finite dimensional $k$ vector spaces; then $F(S)$ is finite over $R$.
\end{proposition} 

\begin{proof}
By Lemma~\ref{l:structure}, $F(S)$ is a quotient of $\oplus_{\lambda \in \Lambda } \Hom_{kG}(M_\lambda,S) \otimes_{E_M} F(M_\lambda)$ as an $R$-module. Since $\Lambda$ is finite, the latter is finite over $R$.
\end{proof}

\section{Depth}
\label{s:depth}

When we mention depth it will always be calculated over some graded local subalgebra $R$ of $S^G$. When $S$ is finite over $R$ and $S_0$ is a field, this is independent of the choice of $R$ and we will omit $R$ from the notation.

In this section we mostly deal with polynomial rings $k[V]$. Our main tool is the following calculation from \cite[4.2]{Sy-Cech}.

Let $P$ be a $p$-group and let $V$ be a finite dimensional $kP$-module. Set $n = \dim_k V$ and $m=\dim_k V^P$. Given homogeneous elements $y_1, \ldots , y_m \in k[V]^P$ such that $k[V^P]$ is finite over $k[y_1, \ldots , y_m]$, there are elements $z_{m+1}, \ldots , z_n \in k[V]^P$ and a $k[z_{m+1}, \ldots , z_n]P$-submodule $U$ of $k[V]$ such that $U$ is finite over $k[z_{m+1}, \ldots , z_n]$ and the natural map $k[y_1, \ldots , y_m] \otimes _k U \rightarrow k[V]$ induced by multiplication is an isomorphism of $k[y_1, \ldots , y_m,z_{m+1}, \ldots , z_n]P$-modules. Set $R=k[y_1, \ldots , y_m,z_{m+1}, \ldots , z_n]$.

\begin{proposition}
\label{p:tensoroff}
For any $k$-linear functor $F$ from $kP$-modules to $k$ vector spaces that commutes with direct sums we have 
\[F(k[V]) \cong k[y_1, \ldots , y_m] \otimes_k F(U)\]
as $R$-modules. 

It follows that $\dim F(k[V]) = \dim_k V^P + \dim F(U)$ and $\depth F(k[V]) = \dim_k V^P + \depth F(U)$, when $F(k[V]) \ne 0$.

If $F$ also takes finite dimensional $kG$-modules to finite dimensional $k$ vector spaces then, in fact, $\depth F(k[V]) \geq \min \{ \dim F(k[V]), \dim_k V^P+1 \}$ when $F(k[V]) \ne 0$.
\end{proposition}

\begin{proof}
This follows immediately from the discussion above, except for the last part. For that, note that the condition on $F$ ensures that $F(U)$ is finite over $R$, by Proposition~\ref{p:finite}. Thus, if $\dim F(U) \geq 1$ then $\depth F(U) \geq 1$.
\end{proof}

\begin{theorem}
\label{th:sub}
Let $G$ be a finite group and $V$ a finite dimensional $kG$-module. Let $M$ be a finite dimensional $kG$-module with vertex $P$. Then $\Hom_{kG}(M,k[V])$ is $0$ if the module of covariants $M_{\ker(V)}$ (the largest quotient on which $\ker(V)$ acts trivially) is $0$. Otherwise it has dimension $\dim_k V$ and depth at least $\min \{ \dim_k V, \dim_k V^P+2 \}$.
\end{theorem}

\begin{proof}
The dimension is clearly at most $\dim_k V$. Any homomorphism from $M$ to $k[V]$ must factor through $M_{\ker(V)}$, so is 0 if $M_{\ker(V)}=0$. Otherwise, by \cite[1.2]{Sy-poly}, $k[V]$ contains a free $k(G/\ker(V))$-module $F$ that is a sum of homogeneous pieces and such that the multiplication map $k[V]^P \otimes F \rightarrow k[V]$ is injective. In particular, $F$ contains the injective hull of a simple quotient of $M_{\ker(V)}$. This gives us a non-zero homomorphism $f \! : M \rightarrow k[V]$.  Multiplication by $f$ embeds $k[V]^P$ into $\Hom_{kG}(M,k[V])$, so the latter must have dimension at least $\dim_k V$. 

For the depth we use Proposition~\ref{prop:vertex} to reduce to the case of a $p$-group and we let $U$ be as in Proposition~\ref{p:tensoroff}. If $\dim \Hom_{kG}(M,U) \leq 1$, we just use Proposition~\ref{p:tensoroff}. If $\dim \Hom_{kG}(M,U) \geq 2$ we use Lemma~\ref{l:2} below on $\Hom_k(M,U)$, which is Cohen-Macaulay.
\end{proof}

\begin{lemma}
\label{l:2}
Let $Z$ be a $k[a,b]G$-module, where $G$ commutes with $a$ and $b$. If $a,b$ is a regular sequence on $Z$ then it is also regular on $Z^G$, provided that $Z^G \ne aZ^G + b Z^G$. If $Z$ is graded and $0$ in large negative degrees and if both $a$ and $b$ are homogeneous of positive degree, then this condition reduces to $Z^G \ne 0$, which is guaranteed if $Z \ne 0$ and $G$ is a $p$-group.
\end{lemma}

\begin{proof}
Since $a$ acts without torsion on $Z$, it also does so on $Z^G$. If $z \in Z^G$ is such that $bz \in aZ^G$ then $z \in aZ$, since $a,b$ is a regular sequence on $Z$; say $z=ay$ for $y \in Z$. For any $g \in G$, $0=(g-1)z = a(g-1)y$, so $(g-1)y=0$. Thus $y \in Z^G$ and $z \in aZ^G$.
The proviso is part of the most common definition of a regular sequence \cite{BH, Eisenbud}. 

For the second part, if $z \in Z$ is non-zero of minimal degree then it cannot be in $aZ^G + b Z^G$. The existence of a regular sequence on $Z$ implies that $Z \ne 0$ and thus $Z^G \ne 0$ if $G$ is a $p$-group.
\end{proof}  

The case $M=k$ of Theorem~\ref{th:sub} is a result of Ellingsrud and Skjelbred \cite{ES}. For some classes of groups, for example cyclic ones, the lower bound on the depth in this case is known to be sharp, but this does not always hold: see \cite{FKS}. In fact, \cite{FKS} contains a somewhat stronger result that we now generalise. Recall that the depth of a module $M$ over a graded Noetherian local ring $R$ can be characterised as the smallest positive integer $i$ such that  $\Ext^i(R/m,M) \ne 0$; see \cite[1.2.8]{BH}, or \cite[9.1.2]{BH}, \cite[proof of 3.14]{Totaro} for the infinitely generated case.

First we generalise Lemma~\ref{l:2}.

\begin{lemma}
\label{l:j}
Given two graded $k$-algebras $A$ and $B$ in non-negative degrees, let $X$ be an $A$-module, $Y$ a $B$-module and $Z$ an $A \otimes_k B$-module (where $A$ and $B$ commute). Suppose that, for some integer $m$, we have $\Ext_A^i(X,\Ext_B^j(Y,Z))=0$ for $i+j \leq m$, $j \geq 1$ and $\Ext_B^i(Y,\Ext_A^j(X,Z))=0$ for $i+j \leq m+1$; then $\Ext^i_A(X,\Hom_B(Y,Z))=0$ for $i \leq m+1$.
\end{lemma}

Note that $\Ext_B^j(Y,Z)$ gets its structure as an $A$-module from the action of $A$ on $Z$, as in Section~\ref{s:functor}.

\begin{proof}
We have natural isomorphisms
\[\Hom_A(X,\Hom_B(Y,Z)) \cong \Hom_{A \otimes_kB}(X \otimes_kY,Z) \cong \Hom_B(Y,\Hom_A(X,Z)),
\]
from which we obtain two spectral sequences converging to $\Ext^*_{A \otimes_kB}(X \otimes_k Y,Z)$.

One has $E_2$-term $E_2^{i,j}=\Ext_B^i(Y,\Ext_A^j(X,Z))$, which is $0$ for $i+j \leq m+1$, so $\Ext^{\ell}_{A \otimes_kB}(X \otimes_k Y,Z)=0$ for $\ell \leq m+1$. The other has all entries above the bottom row and in the region $i+j \leq m$ equal to $0$. Thus there are no non-zero differentials with codomain in $E_*^{i,0}$ for $i \leq m+1$ on this or any subsequent page of the spectral sequence.  Since we know that  $E_{\infty}^{i,0}=0$ for $i \leq m+1$, we must have $0=E_2^{i,0}=\Ext^i_A(X,\Hom_B(Y,Z))$ in this range.
\end{proof}

We now apply this in the case when $S$ is a graded $k$-algebra with action of $G$ and $R \leq S^G$ is a graded local $k$-subalgebra, by setting $A=R$, $B = kG$, $X=R/m$, $Y=M$ and $Z=S$.

\begin{theorem}
\label{t:j}
We have 
\[
\depth_R \Hom_{kG}(M,S) \geq \min \{ \depth_R S, \min_{j \geq 1} \{\depth_R \Ext_{kG}^j(M,S)+j+1 \} \}.
\]
When $S=k[V]$ and $S$ is finite over $R$ this becomes
\[
\depth \Hom_{kG}(M,k[V]) \geq \min \{ \dim_k V, \dim_k V^P + \min_{j \geq 1} \{\depth \Ext_{kG}^j(M,U)+j+1 \} \}.
\]
Here $U$ is as in Proposition~\ref{p:tensoroff} and we interpret the depth of a zero module as $\infty$.
\end{theorem}

\begin{theorem}
\label{th:ext}
Let $G$ be a finite group and $V$ a finite dimensional $kG$-module. Let $M$ be a finite dimensional $kG$-module with vertex $P$. Then
$\Hom^{\mathcal X}_{kG}(M,k[V])$ has dimension at most $\max \{ \dim_k V^Q; Q \in \mathcal S_p(G) \smallsetminus \mathcal X \} $ and if not 0 has depth at least $\dim_k V^P$.

In particular,  $\widehat{\Hom}_{kG}(M,k[V])$ and $\Ext^i_{kG}(M,k[V])$ for $i \geq 1$ have dimension at most $\max \{ \dim_k V^C; C \leq G, |C|=p \} $ and depth at least $\dim_k V^P$ (if it is not 0).
\end{theorem}

\begin{proof}
The result for dimension follows from Corollary~\ref{cor:bound} and that for depth from Propositions~\ref{prop:vertex} and \ref{p:tensoroff}. Observe that $\Ext^i_{kG}(M,k[V]) \cong \widehat{\Hom}_{kG}(\Omega^i M,k[V])$,  where $\Omega$ is the Heller syzygy operator, and note that $\Omega$ preserves the vertex of the module. 
\end{proof}

\begin{theorem}
\label{th:[]}
Let $G$ be a finite group and $V$ a finite dimensional $kG$-module. Let $M$ be a finite dimensional indecomposable $kG$-module with vertex $P$. Then $\Hom_{[kG]}(M,k[V])$ is Cohen-Macaulay of dimension $\dim_k V^P$ or is 0.
\end{theorem}

In the case $M=k$, this was proved by Totaro \cite[9.2]{Totaro} and this case modulo the radical by Fleischmann \cite{Fleischmann}.

\begin{proof}
Again, the upper bound on the dimension follows from Corollary~\ref{cor:bound} and a lower bound on the depth from Propositions~\ref{prop:vertex} and \ref{p:tensoroff}. They are equal.
\end{proof}

\section{Summands}
\label{s:sum}

\begin{theorem}
\label{th:sum}
Let $G$ be a finite group and $V$ a finite dimensional $kG$-module. Let $M$ be a finite dimensional indecomposable $kG$-module with vertex $P$, source $U$ and inertia subgroup $I$.  Then, if it is not zero,  $\Hom^{\oplus}_{kG}(M,k[V])$ has dimension at most $\dim_k V^P$ and depth at least $\dim_k V^{I_p}$.
\end{theorem}

The upper bound on the dimension is from Corollary~\ref{cor:bound}. The proof of the lower bound on the depth will occupy the rest of this section.

\begin{corollary}
\label{c:sum}
In the circumstances of the previous theorem, $\Hom^{\oplus}_{kG}(M,k[V])$ is Cohen-Macaulay if $I_p=P$, in particular if $P$ is a Sylow $p$-subgroup. The ring of trivial summands, $S^{\oplus G}$, is always Cohen-Macaulay.
\end{corollary}

\begin{proposition}
\label{p:inertia}
Let $M$ be a finite dimensional indecomposable $kG$-module with vertex $P$ and inertia group $I$, Let $Y$ be the indecomposable $kI$-module such that $M$ is a summand of $Y \! \uparrow ^G_I$. Then $\Hom^{\oplus}_{kG}(M,-)$ is isomorphic to $\Hom^{\oplus}_{kI}(Y,-\downarrow ^G_I)$ as a functor on $kG$-modules.
\end{proposition}

\begin{proof}
Let $Z$ be the $kN_G(P)$-module that is the Green correspondent of $M$. We will show that both functors are isomorphic to $\Hom^{\oplus}_{kN_G(P)}(Z, -\downarrow ^G_{N_G(P)})$. It is sufficient to check that a natural transformation induces an isomorphism on finite dimensional indecomposable modules, because the functors commute with direct sums, so this gives us the case of all finite dimensional modules, and the functors also commute with filtered colimits, which gives us all modules.

Let $\mathcal X$ be the set of subgroups of $P$ of the form $P \cap \! {}^g \! P$, for $g \in G \smallsetminus N_G(P)$. As an immediate consequence of the functorial version of the Green correspondence \cite[4.1]{Green}, there is
a natural isomorphism $\Hom^{\mathcal X}_{kG}(M,-) \rightarrow \Hom^{\mathcal X}_{kN_G(P)}(Z, -\downarrow ^G_{N_G(P)})$ given by first restricting the group from $G$ to $N_G(P)$ and then restricting the first module
along the inclusion of $Z$ in $M \downarrow ^G_{N_G(P)}$ as a summand. Evaluated at $M$, it yields a ring isomorphism $\End^{\mathcal X}_{kG}(M) \rightarrow \End^{\mathcal X}_{kN_G(P)}(Z)$ and hence $\End^{\oplus}_{kG}(M) \cong \End^{\mathcal X}_{kG}(M)/\rad \End^{\mathcal X}_{kG}(M) \cong \End^{\mathcal X}_{kN_G(P)}(Z)/\rad \End^{\mathcal X}_{kN_G(P)}(Z) \cong \End^{\oplus}_{kN_G(P)}(Z)$. 

Evaluated at any other finite dimensional indecomposable module $N$, clearly $\Hom_{kG}^{\oplus}(M,N)=0$. If $\Hom^{\oplus}_{kN_G(P)}(Z,N \! \downarrow ^G_{N_G(P)}) \ne 0$ then $Z$ is a summand of $N \! \downarrow ^G_{N_G(P)}$, so by the functorial version of the Green correspondence mentioned above, $M$ is a summand of $N$ if we work in the quotient category with morphisms $\Hom_{kG}^{\mathcal X}$. But $\End_{kG}(M)$ is local and $\rad \End_{kG}^{\mathcal X} (M)$ is in its radical, so $M$ is a summand of $N$, a contradiction.

We know that $Z \cong Y \! \uparrow^{N_G(P)}_I$ and $Z \downarrow^{N_G(P)}_I$ has a summand that we identify with $Y$ via inclusion and projection maps $\iota$ and $\rho$. In fact $Z \! \downarrow^{N_G(P)}_I = \oplus_{g \in N_G(P)/I}gY$, where $gY \cong {}^g \! Y$.

The assignment $f \mapsto \rho f \iota$ yields a ring homomorphism $\End_{kN_G(P)}(Z) \rightarrow \End^{\oplus}_{kI}(Y)$. This is because the difference between $\rho f_1 \iota \rho f_2 \iota$ and $\rho f_1 f_2 \iota$ is a sum of morphisms that factor through some ${}^g \! Y$ for $g \not \in N_G(P)$, so are in $K_I(Y,Y)$. This homomorphism is surjective, because any $\overline{f} \in \End^{\oplus}_{kI}(Y)$ lifts to $f \in \End_{kI}(Y)$ and then to $\oplus_{g \in N_G(P)/I} {}^g \! f \in \End_{kN_G(P)}(Z)$. As before, this yields an isomorphism of division algebras $\End^{\oplus}_{kN_G(P)}(Z) \rightarrow \End^{\oplus}_{kI}(Y)$.

Now let $X$ be an indecomposable $kN_G(P)$-module that is not isomorphic to $Z$. Clearly $\Hom^{\oplus}_{kN_G(P)}(Z,X)=0$. Assume that $\Hom^{\oplus}_{kI}(Y,X \! \downarrow^{N_G(P)}_{I}) \ne 0$. Then $X \! \downarrow^{N_G(P)}_{I}$  has a summand $Y'$ that is isomorphic to $Y$ and any $gY' \cong {}^g \! Y'$ for $g \in N_G(P)$ is also a summand. The ${}^g \! Y'$ for $g$ in different cosets modulo $I$ are pairwise non-isomorphic, so $\tilde{Y} = \sum _{g \in N_G(P)/I} gY'$ is actually a direct sum and a $kI$-summand of $X$ (this is a formal consequence of the fact that the indecomposables have local endomorphism rings, cf. the lemma after \cite[19.19]{CR}). It is naturally a $kN_G(P)$-submodule of $X$ and isomorphic to $Z$. It is also projective relative to $I$, so must be a summand of $X$, contradicting the assumption. 

This shows that the two functors are isomorphic on indecomposable modules, as required.
\end{proof}

\begin{proposition}
\label{p:proj}

Let $E$ be an indecomposable projective $kG$-module whose socle $T$ has vertex $P$. Then $\Hom^{\oplus}_{kG}(E,k[V])$ is $0$ if $\ker(V) \not \leq \ker (E)$ and otherwise has dimension equal to $\dim_k V$ and depth at least $\min \{ \dim_k V, \dim_k V^P+2, \dim \widehat{\Hom}_{kG}(T,k[V]) +1 \} \geq \min \{ \dim_k V, \dim_k V^P +1 \}$.

In particular, if $G$ is a $p$-group $P$ then $T=k$ and the depth is at least $\min \{ \dim_k V, \dim_k V^P+2, \max_{C \leq P, |C|=p} \dim_k V^C +1 \}$.
\end{proposition}

\begin{proof}
By the condition on the kernels, we can work over $G/\ker(V)$ so assume $V$ is faithful. The dimension is then found as in the proof of Theorem~\ref{th:sub}, because the $E$ here is isomorphic to a summand of the $F$ there.

Any non-zero homomorphism $T \rightarrow X$ is injective, since $T$ is simple; if it factors through a projective module then it also factors through through an injective module, since projective $kG$-modules are also injective. Thus it factors through the injective hull of $T$, which is $E$. The map $E \rightarrow X$ is injective, since it is injective on its socle, and its image is a summand, since $E$ is an injective module. It follows that we have a short exact sequence
\[ 0 \rightarrow \Hom^{\oplus}_{kG}(E, -) \rightarrow \Hom_{kG}(T,-) \rightarrow \widehat{\Hom}_{kG}(T,-) \rightarrow 0.
\]

 By a standard property of depth on short exact sequences \cite[18.6]{Eisenbud}
 \cite[1.2.9]{BH}
\[
\depth \Hom^{\oplus}_{kG}(E,k[V]) \geq \min \{ \depth \Hom_{kG}(T,k[V]), \depth \widehat{\Hom}_{kG}(T,k[V])+1\}
\]
 (with the usual provisos when a module is $0$).

But $\depth \Hom_{kG}(T,k[V]) \geq \min \{ \dim_k V, \dim_k V^P+2 \}$ by Theorem~\ref{th:sub} and $\depth \widehat{\Hom}_{kG}(T,k[V]) \geq \dim_k V^P$, by Theorem~\ref{th:ext}.

For the $p$-group case, use Corollary~\ref{cor:bound}.
\end{proof}

\begin{proposition}
\label{p:red}
Let $M$ be a finite dimensional indecomposable $kG$-module that is not projective. Then there is an additive functor $F$ from $kG_p$-modules to $k$-modules such that $\Hom^{\oplus}_{kG}(M,-)$ is a summand of $F(-\downarrow^G_{G_p})$ as functors.
\end{proposition}

\begin{proof}
Let \[
0 \rightarrow M \rightarrow L \rightarrow K \rightarrow 0
\]
be the almost split sequence beginning with $M$. According to  \cite[4.12.6]{Benson-rep}, there is an exact sequence of functors 
\[
0 \rightarrow \Hom_{kG}(K,-) \rightarrow \Hom_{kG}(L,-) \rightarrow \Hom_{kG}(M,-) 
\rightarrow \Hom^{\oplus}_{kG}(M,-)  \rightarrow 0.
\]
A priori this is just as functors on finite dimensional modules, but all these functors commute with filtered colimits by Lemma~\ref{l:colim} and filtered colimits are exact in this context. We obtain a diagram

\begin{adjustbox}{width=\textwidth,center}
$
\xymatrix{
0  \ar[r]  & \Hom_{kG}(K,-) \ar@/_/[d]_r \ar[r]  & \Hom_{kG}(L,-) \ar@/_/[d]_r \ar[r] & \Hom_{kG}(M,-) 
\ar@/_/[d]_r \ar[r] & \Hom^{\oplus}_{kG}(M,-)  \ar[r] & 0 \\
0  \ar[r]  &\Hom_{kG_p}(K,-\downarrow^G_{G_p}) \ar@/_/[u]_t \ar[r] &\Hom_{kG_p}(L,-\downarrow^G_{G_p}) \ar@/_/[u]_t \ar[r]^f &\Hom_{kG_p}(M,-\downarrow^G_{G_p}) \ar@/_/[u]_t &&
,
}
$
\end{adjustbox}
where $r$ stands for restriction and $t$ for $|G:G_p|^{-1}\tr^G_{G_p}$. The squares commute and $ t \circ r=\Id$. It follows that $\Hom^{\oplus}_{kG}(M,-)$ embeds as a functor as a summand of $\Coker(f)$.
\end{proof}

\begin{lemma}
\label{l:plusdepth}
For any indecomposable $kG$-module $M$, $\Hom^{\oplus}_{kG}(M,k[V])$ is $0$ or has depth at least $\dim_k V^{G_p}$. 
\end{lemma}

\begin{proof}
Combine Proposition~\ref{p:tensoroff} and Proposition~\ref{p:red}.
\end{proof}

Theorem~\ref{th:sum} can now be proved for non-projective modules by using Proposition~\ref{p:inertia} to reduce to $I$ and then Lemma~\ref{l:plusdepth} to reduce to $I_p$.
For projective modules it is a weak form of Proposition~\ref{p:proj}.

\section{Dickson Invariants}
\label{s:Dickson}

In Theorem~\ref{t:equiv}, if we find a sub-algebra $Q$ of $R$ such that $S/T$ is still finite over $Q$, then the other rings are also finite over $Q$ and so are any finitely generated modules for these rings, such as occur in Corollary~\ref{cor:bound}.

When $k=\mathbb F_q$, the field of $q$ elements, $q$ a power of $p$,  $V$ an $\mathbb F_q$ vector space of dimension $n$ and $G \subseteq \Gl_{\mathbb F_q}(V)$, we have $\mathbb F_q [c_{n,0}, \ldots , c_{n,n-1}] \subseteq \mathbb F_q[V]^G$, where the $c_{n,i}$ are the Dickson invariants. For any $\mathbb F_q$-subspace $W \leq V$ of dimension $m$, the subset of elements $c_{n,i}$ with $i \geq n-m$ forms a system of parameters on $W$ \cite[8.1.2]{Benson-inv}.

In Proposition~\ref{p:tensoroff}, we can use $c_{n,n-m}, \ldots , c_{n,n-1}$ as $y_1, \ldots , y_m$.
In particular, in the context of Theorem~\ref{th:[]}, $\Hom_{[kG]}(M,\mathbb F_q[V])$ is finite and free over $\mathbb F_q[c_{n,n-m}, \ldots , c_{n,n-1}]$, for $n=\dim_k V$ and $m=\dim_k V^P$. Similarly, in the context of Theorem~\ref{th:sum}, $\Hom^{\oplus}_{kG}(M,\mathbb F_q[V])$ is finite over $\mathbb F_q[c_{n,n-m}, \ldots , c_{n,n-1}]$ and free (but maybe not finite) over $\mathbb F_q[c_{n,n- \ell}, \ldots , c_{n,n-1}]$, where $\ell = \dim_k V^{I_p}$.

\section{The Non-Graded Case}
\label{s:ng}

Most of our results for general $S$ also apply in the non-graded case with the same proofs, either by ignoring the grading or by putting eveything in degree 0. Of course, regularity no longer makes sense and where depth is mentioned we must assume that $R$ is local or replace depth with the grade with respect to an ideal.

The results for polynomial rings $k[V]$ also hold if we take the graded case and just forget the grading. We replace depth by grade with respect to what used to be the graded maximal ideal. For more general group actions on a polynomial ring our methods do not apply,because they depend on Proposition~\ref{p:tensoroff}. 

\section{Examples}
\label{s:examples}

Let $p=2$, $C$ a group of order 2 and $V$ a free $kC$-module of rank 1, where $C$ swaps the basis elements $x$ and $y$ and their duals $x^*$ and $y^*$. Then $k[V]^C=k[a,b]$, where $a=x^*+y^*$ and $b=x^*y^*$.

We can identify $\Hom^\oplus_{kC}(V,S)$ with the image of $\tr^C_1$ on $S$ (there are only two indecomposable modules and if we substitute either one for $S$ the claim is true). 

Consider the sum of two copies of $V$, $V_1 \oplus V_2$. We know that $k[V_1 \oplus V_2]$ is finite over $k[a_1,a_2,b_1,b_2]$. Clearly $\tr^C_1k[V_1 \oplus V_2]$ has two $k$-basis elements in degree 1, $a_1$ and $a_2$, and $a_1a_2=a_2a_1$ (where in each product the first $a_i$ is considered to be a parameter and the second a basis element), so $\tr^C_1k[V_1 \oplus V_2]$ cannot be Cohen-Macaulay. It must have dimension 4 and depth 3, by Proposition~\ref{p:proj}.

Now let $p=3$, D a cyclic group of order 3 and $V$ an indecomposable $kD$-module of dimension 2.
Then $k[V]^D=k[c,d]$, where $\deg c=1$ and $\deg d=3$. The structure as a $kD$-module is known (cf.\ \cite[1.1]{KS}): $S=k[V] \cong k[d] \otimes ( S_0 \oplus S_1 \oplus (k[c] \otimes S_2))$. This leads to a Hilbert Series with coefficients in the Green ring (in fairly obvious notation, but cf.\ \cite{HS}) 
\[
\mathcal H (S,t)= \frac{1}{1-t^3}\left( [k] + t [V] + \frac{t^2}{1-t}[F] \right),
\]
where $F$ denotes the free module of rank 1.

Now consider the sum of three copies of $V$. A straightforward calculation shows that 
\begin{multline*}
\mathcal H(k[V_1 \oplus V_2 \oplus V_3],t) = \mathcal H(S^{\otimes 3},t) = \mathcal H(S,t)^3 = \\ = \frac{1}{(1-t^3)^3}\left( (1+3t^2)[k] + (3t+t^3)[V] + \frac{6t^2-t^3+3t^4+t^6}{(1-t)^3}[F] \right).
\end{multline*}
We know that the $[k]$ and $[V]$ parts are Cohen-Macaulay of dimension 3, by Proposition~\ref{p:tensoroff} or Corollary~\ref{c:sum}. But the $[F]$ part cannot be Cohen-Macaulay, because the terms in the denominator of the Hilbert series correspond to the degrees of the obvious set of parameters but there is a minus sign in the numerator. It must have dimension 6 and depth 4 or 5, by Proposition~\ref{p:proj}. The $[k]$ part corresponds to $k[V_1 \oplus V_2 \oplus V_3]^{\oplus D}$; this is not Gorenstein, again by considering the numerator in the Hilbert series, in contrast to the result in \cite{Stanley} that in characteristic 0 the invariant subalgebra is Gorenstein when the determinant of the representation $V$ is trivial.

In fact $k[V_1 \oplus V_2 \oplus V_3]^{\oplus D}$ must be free over $k[d_1,d_2,d_3]$, with basis 1 and three elements in degree 2 (which can be taken to be the images of elements of the form $x_iy_j-x_jy_i$). The product of any two of these images must be $0$, because there is nothing in degree 4, so we have completely determined the algebra $k[V_1 \oplus V_2 \oplus V_3]^{\oplus D}$. 

The detailed calculations for cyclic groups in \cite{Elmer} can be used to construct many similar examples.
Larger examples are calculated in \cite{KS, KS1}. Although they are misleadingly well behaved, they were the motivation for this work.

\end{document}